\documentclass[11pt]{amsart}

\usepackage{amsaddr}

\title{Partial Category Actions on Sets and Topological Spaces}

\newtheorem{thm}{Theorem}
\newtheorem{prop}[thm]{Proposition}
\theoremstyle{definition}
\newtheorem{defi}[thm]{Definition}
\newtheorem{exa}[thm]{Example}
\newtheorem{rem}[thm]{Remark}

\begin{document}

\author{Patrik Nystedt}

\address{University West,
Department of Engineering Science, 
SE-46186 Trollh\"{a}ttan, Sweden}

\email{Patrik.Nystedt@hv.se}



\begin{abstract}
We introduce (continuous) partial category actions on sets 
(topological spaces) and
show that each such action admits a universal globalization.
Thereby, we obtain a simultaneous generalization of 
corresponding results for groups, by F. Abadie,
and J. Kellendonk and M. V. Lawson,
and for monoids, by M. G. Megrelishvili and L. Schr\"{o}der.
We apply this result to the special case of partial groupoid actions 
where we obtain a sharpening of a result 
by N. D. Gilbert, concerning ordered groupoids,
in the sense that mediating functions between 
universal globalizations always are injective.
\end{abstract}

\maketitle

\section{Introduction}

Partial group actions on sets 
is a relaxation of classical group actions and was
introduced at the end of the last century by R. Exel \cite{Exel98}, 
with impetus coming from questions originating in the context of partial
actions of groups on C*-algebras 
(see for instance \cite{abadie1998}, \cite{exel1994} or \cite{mcclanahan}).
Since then, partial group actions on sets have  
appeared in many different contexts.
Indeed, Kellendonk and Lawson \cite{kellendonk} 
have shown that they appear naturally in the theories
of ${\Bbb R}$-trees, tilings and model theory. 
In loc. cit. it was also pointed out that
the M\"{o}bius group acts globally on the Riemann sphere but
only partially on the complex plane.
Another application of partial group actions on sets
was given by Birget in \cite{birget2004} where it was shown that
Thompson's group can be defined via a partial group
action on finite words.
For examples of other types of partial actions and applications
of these in different contexts,
see Dokuchaev's extensive survey article \cite{dokuchaev2011}.

Let us recall the definition of partial group actions on sets.
Let $G$ be a group with identity element $e$
and let $X$ be a set. 
A {\it partial group action} of $G$ on $X$
is a partial function $G \times X \rightarrow X$,
denoted by $G \times X \ni (g,x) \mapsto g \cdot x$,
for all $g \in G$ and all $x \in X$ such that $g \cdot x$ is defined,
satisfying the following three axioms.
\begin{itemize}

\item[(G1)] If $x \in X$, then $e \cdot x$ is defined 
and equal to $x$.

\item[(G2)] If $x \in X$ and $g \in G$ are chosen so that 
$g \cdot x$ is defined, then $g^{-1} \cdot (g \cdot x)$
is defined and equal to $x$. 

\item[(G3)] If $x \in X$ and $g,h \in G$ are chosen so that 
$g \cdot (h \cdot x)$ is defined, then 
$(gh) \cdot x$ is defined and equal to 
$g \cdot (h \cdot x)$.

\end{itemize}
Following \cite{Exel98}, every
such a partial action can equivalently
be formulated in terms of a pair 
$( \{ D_g \}_{g \in G} , \{ \alpha_g \}_{g \in G} \} ),$
where for each $g \in G$, $D_g$ is a subset of $X$
and $\alpha_g : D_{g^{-1}} \rightarrow D_g$ is a bijection
satisfying the following three axioms.
\begin{itemize}

\item[(G1$'$)] $\alpha_e = {\rm id}_X$.

\item[(G2$'$)] If $g,h \in G$, then 
$\alpha_g (D_{g^{-1}} \cap D_h) = D_g \cap D_{gh}$.

\item[(G3$'$)] If $g,h \in G$ and $x \in D_{h^{-1}} \cap D_{(gh)^{-1}}$,
then $\alpha_g ( \alpha_h (x) ) = \alpha_{gh}(x)$.

\end{itemize}
If the partial function $G \times X \rightarrow X$
is, indeed, a function, then the group action is called {\it global}.
Given a global group action by $G$ on a set $Y$, then 
it is easy to construct partial group actions 
on subsets of $Y$ by {\it restriction}.
In fact, suppose that $X$ is a subset of $Y$.
Then we can define a partial group action by $G$ on $X$
in the following way.
If $g \in G$ and $x \in X$, then $g \cdot x$ is defined
if and only if $g \cdot x \in X$.
F. Abadie \cite{abadie2003} (and independently 
J. Kellendonk and M. Lawson \cite{kellendonk}) has shown that 
restriction in fact provides us with {\it all} 
examples of partial group actions
(see Theorem \ref{kellendonktheorem}).
Namely, suppose that $X$ and $Y$ are sets both of which 
are equipped with a partial group action by $G$.
A function $i : X \rightarrow Y$ is called a {\it $G$-function}
(or {\it morphism of partial actions} as defined in \cite{abadie2003})
if for every $g \in G$ and every $x \in X$,
such that $g \cdot x$ is defined, $g \cdot i(x)$
is also defined and $i(g \cdot x) = g \cdot i(x)$.
In that case, if the group action by 
$G$ on $Y$ is global, then 
$Y$ is called a {\it globalization} of $X$.
Such a globalization is called {\it universal}
if for every $G$-function $j : X \rightarrow Z$,
where $Z$ is a set equipped with a global group 
action by $G$, there is a unique $G$-function
$k : Y \rightarrow Z$ which is {\it mediating} i.e. 
with the property that $j = k \circ i$.

\begin{thm}[F. Abadie \cite{abadie2003}]\label{kellendonktheorem}
Every partial group action on a set admits a universal globalization.
\end{thm}

Moreover, it follows from \cite[Theorem 3.4]{kellendonk} 
that the universal globalization $Y$ can be chosen so that
$i : X \rightarrow Y$ is injective, the global action by $G$ on 
$i(X)$ induces the original partial action on $X$ and that
$Y$ is the best possible globalization of $X$ in the sense that
if we are given another injective globalization 
$j : X \rightarrow Z$, then the mediating $G$-function 
$k : Y \rightarrow Z$, from above, is injective.

Continuous partial actions of groups on topological spaces
and the existence of universal globalizations
have been studied independently by 
Abadie \cite{abadie2003}, and Kellendonk and Lawson \cite{kellendonk}.
In this context $G$ is topological group, 
$X$ is a topological space, the set
$\Gamma = \{ (g,x) \in G \times X \mid g \cdot x \ {\rm is} \
{\rm defined} \}$ is open in $G \times X$,
the action $\Gamma \rightarrow X$
is continuous and for every $g \in G$, $D_g$ is open in $X$ and 
$\alpha_g : D_{g^{-1}} \rightarrow D_g$ 
is a homeomorphism.

When discussing generalizations of partial group actions, 
it is natural to formulate the following general questions.
\begin{itemize}

\item[(Q1)] If $G$ is any set equipped with a, possibly partial, composition law,
can we define what it should mean for $G$ to act partially on a set $X$?

\item[(Q2)] Given such a definition, is there a universal globalization?

\end{itemize}
M. G. Megrelishvili and L. Schr\"{o}der \cite{megrelishvili}
have suggested an answer to (Q1) when $G$ is a {\it monoid}
and they give an affirmative answer to (Q2) for such $G$
(see Theorem \ref{megrelishvilitheorem}).
Indeed, let $G$ be a monoid with identity element $e$ and let $X$ be a set.
Following \cite{megrelishvili}, we say that
a partial monoid action of $G$ on $X$
is a partial function $G \times X \rightarrow X$
satisfying the following two axioms.
\begin{itemize}

\item[(M1)] If $x \in X$, then $e \cdot x$ is defined 
and equal to $x$.

\item[(M2)] If $x \in X$ and $g,h \in G$ are chosen so that
$h \cdot x$ is defined, then 
$(gh) \cdot x$ is defined if and only if $g \cdot (h \cdot x)$
is defined, and, in that case, $(gh) \cdot x = g \cdot (h \cdot x)$.

\end{itemize}
Note that the partial monoid actions of a group are 
precisely its partial group actions.
In other words, if the monoid $G$ is a group,
then (G1)-(G3) hold if and only if (M1)-(M2) hold
(see \cite[Proposition 2.4]{megrelishvili}). 
The concepts of $G$-functions and (universal) globalizations
for partial monoid actions are defined precisely as in the 
case of partial group actions (see \cite[Definition 2.5]{megrelishvili}).

\begin{thm}[M. G. Megrelishvili and L. Schr\"{o}der \cite{megrelishvili}]\label{megrelishvilitheorem}
Every partial monoid action on a set admits a universal globalization.
\end{thm}

Note that, in \cite[Section 3]{megrelishvili} a globalization result for continuous
partial monoid actions is obtained.
However, there, a partial continuous monoid 
action on $X$ by $G$ presupposes that $G$
has the {\it discrete} topology.
Note also that C. Hollings \cite[Definition 2.2]{hollings}  
has defined the notion of {\it weak}
partial monoid action.
We make no attempt in this article to relate to that type of action.
We just remark that C. Hollings \cite[Theorem 5.7]{hollings}
has shown that a weak partial monoid action can be globalized
if and only if the action is a partial monoid action
in the sense of Megrelishvili and Schr\"{o}der \cite{megrelishvili}.

N. D. Gilbert \cite{gilbert} has suggested an answer to (Q1)
when $G$ is an {\it ordered groupoid}. 
He has also given an affirmative answer to (Q2) for such $G$.
A partial action in this context is an ordered premorphism
$G \rightarrow {\Bbb G}(X)$, where ${\Bbb G}(X)$ is the symmetric
groupoid of $X$ whose identities are the identity maps
on subsets of $X$ and whose morphisms are bijections
between subsets of $X$.  

\begin{thm}[N. D. Gilbert \cite{gilbert}]\label{gilberttheorem}
Every ordered groupoid action on a set admits a universal globalization.
\end{thm}

Note, however, that N. D. Gilbert loc. cit. is not able to
extend all results for partial group actions
to partial groupoid actions. Namely, given 
another globalization $Z$ of $X$, then he only shows
that the mediating $G$-function 
is {\it partially injective} 
(see \cite[Proposition 4.10]{gilbert} and Remark \ref{comparegilbert}).

In this article, we suggest an answer to (Q1)
when $G$ is an arbitrary {\it category} 
and we give an affirmative answer to (Q2) in that case.
Thereby, we simultaneously generalize
Theorem \ref{kellendonktheorem} and
Theorem \ref{megrelishvilitheorem}, and,
in the special case of groupoid actions,
we obtain a sharpening of
Theorem \ref{gilberttheorem} in the sense 
that mediating $G$-functions between universal 
globalizations are injective.
We also show that the all of the results concerning 
topological globalizations of partial actions, mentioned
above, can be generalized to the context of
continuous partial actions on topological spaces
by topological categories.

Here is a detailed outline of the article.

In Section \ref{sectionpartialcategoryactions},
we state our conventions on categories
and we define what it should mean for a 
category $G$ to act partially (or globally) on a set $X$
(see Definition \ref{definition}).
In the same section, we give a description 
of partial category actions via subsets of $X$ and 
partially defined maps on $X$ (see Proposition \ref{alternative}
and Proposition \ref{alternative'}) analogous to 
the axioms (G1$'$)-(G3$'$).
We also show that global category actions correspond
to functors $G \rightarrow {\rm Set}$ 
(see Proposition \ref{remarkfunctor}).

In Section \ref{sectionglobalization}, 
we introduce the concepts of $G$-functions and 
universal globalization of partial actions of categories on sets
(see Definition \ref{equivariant}), and
we show the following result.

\begin{thm}\label{maintheorem}
Every partial category action on a set admits a universal globa\-lization.
\end{thm}

In the same section, we show that $X$ injects into a 
universal globalization $Y$ so that the global category action 
of $G$ on the image of $X$ in $Y$ induces the original partial
action of $G$ on $X$ (see Remark \ref{injectivecategory}).

In Section \ref{sectionpartialgroupoidactions},
we state our conventions on groupoids and we define
what it should mean for a groupoid to act partially on a set
(see Definition \ref{definitiongroupoid}).
Then we show that the the partial category actions of a 
groupoid are precisely its partial groupoid actions
(see Proposition \ref{categorygroupoid}).
Thereafter we show, as a generalization of \cite[Theorem 3.4]{kellendonk}, 
that mediating $G$-functions between universal globalizations
always are injective (see Proposition \ref{sharpening}).
At the end of this section, we
compare our definition with the one
given by N. D. Gilbert \cite{gilbert} for partial actions of 
ordered groupoids on sets (see Remark \ref{comparegilbert}).

The results of Sections \ref{sectionpartialcategoryactions}-\ref{sectionpartialgroupoidactions}
are a preparation for Section \ref{sectioncontinuous}, where we
state our conventions on continuous partial functions and topological categories,
and we define what it should mean for a 
topological category $G$ to act continuously partially 
on a topological space $X$
(see Definition \ref{definitioncontinuousaction}).
In this section, we show the following result.

\begin{thm}\label{continuousmaintheorem}
Every continuous partial category action on a
topological space admits a universal globalization.
\end{thm}

At the end of this section, we also show that if the category is star open
and the action is graph open (see Definition \ref{definitionstaropen}),
then the embedding of $X$, into a universal globalization,
is an open map (see Proposition \ref{open}).

\section{Partial Category Actions}\label{sectionpartialcategoryactions}

Throughout this section, $X$ denotes a set
and $G$ denotes a category.
In this section,
we state our conventions on categories and 
then we define partial set actions 
(see Definition \ref{setaction}).
Inspired by the axioms (M1)-(M2), from the 
case of partial actions of monoids on sets,
we then define partial category actions on sets
(see Definition \ref{definition}).
Thereafter, we give a description 
of partial category actions via subsets and 
partially defined maps (see Proposition \ref{alternative}
and Proposition \ref{alternative'})
analogous to the group axioms (G1$'$)-(G3$'$)
from the introduction.
At the end of this section,
we show that global category actions correspond
to set-valued functors (Proposition \ref{remarkfunctor}).

\subsection*{Conventions on categories}
The family of objects and morphisms
of $G$ is denoted by ${\rm ob}(G)$ and ${\rm mor}(G)$ respectively. 
We always assume that $G$ is {\it small} i.e.
that ${\rm mor}(G)$ is a set.
By abuse of notation,
we identify an object $e \in {\rm ob}(G)$ 
with its corresponding identity morphism,
so that ${\rm ob}(G) \subseteq {\rm mor}(G)$.
If $g \in {\rm mor}(G)$, then the domain and codomain of 
$g$ is denoted by $d(g)$ and $c(g),$ respectively.
We let $G^2$ denote the set of all pairs $(g,h) \in {\rm mor}(G) \times {\rm mor}(G)$
that are composable i.e. such that $d(g)=c(h)$.
The {\it category of sets} is denoted by Set.

\begin{defi}\label{setaction}
Suppose that $A$, $B$, $C$, $X$ and $Y$ are sets.
By a {\it partial function} $A \rightarrow B$, 
we mean a function $C \rightarrow B$ where $C \subseteq A$.
By a {\it partial set action} of $A$ on $X$,
we mean a partial function $A \times X \rightarrow X$,
denoted by $A \times X \ni (a,x) \mapsto a \cdot x$,
for all $a \in A$ and all $x \in X$ such that $a \cdot x$ is defined.
A function $i : X \rightarrow Y$ is called an {\it $A$-function}
if for every $a \in A$ and every $x \in X$,
such that $a \cdot x$ is defined, $a \cdot i(x)$
is also defined and $i(a \cdot x) = a \cdot i(x)$.
\end{defi}

\begin{defi}\label{definition}
By a {\it partial category action by $G$ on $X$}, 
we mean a partial set action by ${\rm mor}(G)$ on $X$,
in the sense of Definition \ref{setaction},
satisfying the following three axioms.
\begin{itemize}

\item[(C1)] For every $x \in X$, there is $e \in {\rm ob}(G)$
such that $e \cdot x$ is defined. 
If $f \in {\rm ob}(G)$ and $x \in X$ are chosen so that
$f \cdot x$ is defined, then $f \cdot x = x$.

\item[(C2)] If $x \in X$ and $g \in {\rm mor}(G)$
are chosen so that $g \cdot x$ is defined, then
$d(g) \cdot x$ is defined.

\item[(C3)] Suppose that $(g,h) \in G^2$ and $x \in X$ are chosen so that
$h \cdot x$ is defined. 
Then $(gh) \cdot x$ is defined if and only if
$g \cdot (h \cdot x)$ is defined, and, in that case,
$(gh) \cdot x = g \cdot (h \cdot x)$.

\end{itemize}
We say that such an action by $G$ on $X$ is {\it global}
if the following axiom holds.
\begin{itemize}

\item[(C4)] If $g \in {\rm mor}(G)$ and $x \in X$
are chosen so that $d(g) \cdot x$ is defined,
then $g \cdot x$ is defined.

\end{itemize}
\end{defi}

Suppose that $X$ is equipped with a partial
category action by $G$.
Take $g \in {\rm mor}(G)$. 
Put $X_g = \{ x \in X \mid \mbox{ $g \cdot x$ is defined } \}$
and ${}_g X = \{ g \cdot x \mid x \in X_g \}.$
If we define the function $\alpha_g : X_g \rightarrow {}_g X$
by $\alpha_g(x) = g \cdot x$, for $x \in X_g$,
then this defines a triple 
\begin{equation}\label{triple}
( \{ X_g \}_{g \in {\rm mor}(G)} , 
\{ {}_g X \}_{g \in {\rm mor}(G)} , 
\{ \alpha_g \}_{g \in {\rm mor}(G)} ),
\end{equation}
of subsets $X_g$ and ${}_g X$ of $X$ and 
functions $\alpha_g : X_g \rightarrow {}_g X$, 
for $g \in {\rm mor}(G)$.

\begin{prop}\label{alternative}
The triple (\ref{triple}) satisfies the following three properties.
\begin{itemize}

\item[(C1$'$)] $X = \cup_{e \in {\rm ob}(G)} X_e$ and
$\alpha_e = {\rm id}_{X_e}$, for $e \in {\rm ob}(G)$.

\item[(C2$'$)] If $g \in {\rm mor}(G)$, then $X_g \subseteq X_{d(g)}$.

\item[(C3$'$)] Suppose that $(g,h) \in G^2$. Then 
$X_h \cap X_{gh} = 
\alpha_h^{-1}(X_g \cap {}_h X)$.
If $x \in X_h \cap X_{gh}$, then 
$\alpha_g(\alpha_h (x)) = \alpha_{gh}(x)$.

\end{itemize}
If the action by $G$ on $X$ is global,
then the following property holds.
\begin{itemize}

\item[(C4$'$)] If $g \in {\rm mor}(G)$, then $X_g = X_{d(g)}$. 

\end{itemize}
\end{prop}

\begin{proof}
(C1$'$) and (C2$'$) follow immediately from (C1) and (C2).
Now we show (C3$'$). Take $(g,h) \in G^2$ and $x \in X_h$.
From (C3) it follows that 
$(gh) \cdot x$ is defined if and only if $g \cdot (h \cdot x)$
is defined. This means that 
$x \in X_{gh}$ if and only if $\alpha_h(x) \in X_g$.
This implies that 
$x \in X_{gh}$ if and only if $x \in \alpha_h^{-1}(X_g \cap {}_h X)$.
Hence $x \in X_h \cap X_{gh}$
if and only if $\alpha_h^{-1}(X_g \cap {}_h X)$.
From (C3) it follows that for such $x$,
the equality $\alpha_g(\alpha_h(x)) = \alpha_{gh}(x)$ holds.

Suppose that the action by $G$ on $X$ is global.
We show (C4$'$). Take $g \in {\rm mor}(G)$.
From (C2$'$) we know that $X_g \subseteq X_{d(g)}$.
On the other hand,
from (C4) it follows that $X_{d(g)} \subseteq X_g$.
Thus $X_g = X_{d(g)}$. 
\end{proof}

\begin{prop}\label{alternative'}
Starting with a triple (\ref{triple}) 
satisfying (C1$'$)-(C3$'$), then the partial set action by
${\rm mor}(G)$ on $X$ defined by $g \cdot x = \alpha_g(x)$, 
for $g \in {\rm mor}(G)$ and $x \in X_g$, is a 
partial category action by $G$ on $X$.
In that case, if (C4$'$) holds,
then this action is global.
\end{prop}

\begin{proof}
First we show (C1).
Take $x \in X$. Since $X = \cup_{e \in {\rm ob}(G)} X_e$,
there is $e \in {\rm ob}(G)$ such that $x \in X_e$.
Since $\alpha_{e} = {\rm id}_{X_{e}}$ we get that 
$e \cdot x$ is defined and equal to $x$.

Now we show (C2).
Take $x \in X$ and $g \in {\rm mor}(G)$ such that 
$g \cdot x$ is defined i.e. such that $x \in X_g$.
From (C2$'$) we get that $x \in X_{d(g)}$.
Thus $d(g) \cdot x$ is defined.

Now we show (C3).
Take $(g,h) \in G^2$ and $x \in X$ so that
$h \cdot x$ is defined i.e. such that $x \in X_h$.
Suppose that $(gh) \cdot x$ is defined.
Then $x \in X_h \cap X_{gh}$. 
From (C3$'$) we get that $x \in \alpha_h^{-1}(X_g \cap {}_h X)$
which in turn implies that $g \cdot (h \cdot x)$ is defined.
Suppose now that $g \cdot (h \cdot x)$ is defined.
Then $x \in \alpha_h^{-1}(X_g \cap {}_h X)$.
From (C3$'$) we get that
$x \in X_h \cap X_{gh}$ which means that
$(gh) \cdot x$ is defined.

Suppose now that (C4$'$) hold.
We show that (C4) holds.
Take $g \in {\rm mor}(G)$ and $x \in X$
so that $d(g) \cdot x$ is defined.
This means that $x \in X_{d(g)}$.
From (C4$'$), we get that $x \in X_g$.
Thus $\alpha_g(x)$ is defined.
Hence $g \cdot x$ is defined.
\end{proof}

\begin{prop}\label{remarkfunctor}
Global actions
by $G$ on $X$ correspond to 
functors $G \rightarrow {\rm Set}$.
\end{prop}

\begin{proof}
Suppose that $X$ is equipped with a global category action by $G$.
Take $(g,h) \in G^2$. From (C3$'$) and (C4$'$), we get that 
$X_{d(h)} \subseteq \alpha_h^{-1}( X_{c(h)} \cap {}_h X )$.
From this inclusion, we get that
$\alpha_h( X_{d(h)} ) \subseteq X_{c(h)}$.
Since $c(h) = d(g)$, 
$X_{c(h)} = X_{d(g)}$. Also, from (C3$'$),
it follows that $\alpha_g (\alpha_h (x) ) = \alpha_{gh}(x)$,
for $x \in X_{d(h)}$.
It is now evident that the correspondence
$G \rightarrow {\rm Set}$, defined by ${\rm ob}(G) \ni e \mapsto X_e$
and ${\rm mor}(G) \ni g \mapsto \alpha_g$ defines a functor.

On the other hand, suppose that we are given a functor
$F : G \rightarrow {\rm Set}$.
For every $g \in {\rm mor}(G)$,
put $X_g = F( d(g) )$, ${}_g X = F( c(g) )$ and $\alpha_g = F(g)$.
Now put $X = \cup_{e \in {\rm ob}(G)} X_e$.
Then this defines a global action by $G$ on $X$.
\end{proof}

\section{Globalization of Partial Category Actions}\label{sectionglobalization}

Throughout this section, $X$ denotes a set and $G$ denotes a category
which acts partially on $X$.
In this section, we introduce the concepts of
$G$-functions and universal globalization of partial actions of categories on sets
(see Definition \ref{equivariant}), and
we show Theorem \ref{maintheorem}.
At the end of this section, we show that $X$ injects into a 
universal globalization $Y$ so that the global category action 
of $G$ on the image of $X$ in $Y$ induces the original partial
action of $G$ on $X$ (see Remark \ref{injectivecategory}).

\begin{defi}\label{equivariant}
Suppose that there is a partial category action by $G$
on another set $Y$.
We say that a function $i : X \rightarrow Y$ is 
a $G$-function if it is a ${\rm mor}(G)$-function
in the sense of Definition \ref{setaction}. 
In that case, if 
the category action by 
$G$ on $Y$ is global, then 
$Y$ is called a {\it globalization} of $X$.
We say that such a globalization {\it universal}
if for every 
$G$-function $j : X \rightarrow Z$,
where $Z$ is a set equipped with a global category 
action by $G$, there is a unique $G$-function
$k : Y \rightarrow Z$ such that $j = k \circ i$.
\end{defi}

Inspired by the relation $\sim$ defined for group
actions \cite{kellendonk}, for monoid actions \cite{megrelishvili}
and for actions by ordered groupoids \cite{gilbert},
we suggest the following.

\begin{defi}\label{definitionsim}
Put $\overline{X} = \{ (g,x) \in {\rm mor}(G) \times X \mid 
\mbox{ $d(g) \cdot x$ is defined}
\}$ and
define the relation $\sim$ on $\overline{X}$
in the following way.
Suppose that $(g,x),(g',x') \in \overline{X}$.
Then we put $(g,x) \sim (g',x')$ if either
\begin{itemize}
\item[(i)] there is $h \in {\rm mor}(G)$
such that $(g',h) \in G^2$, $h \cdot x$ is defined and 
the equalities $g = g' h$ and $x' = h \cdot x$ hold, or
\item[(ii)] $x=x'$, $g,g' \in {\rm ob}(G)$
and both $g \cdot x$ and $g' \cdot x'$ are defined.
\end{itemize}
\end{defi}

\begin{rem}\label{comparesim}
Note that our relation $\sim$, on account of 
condition (ii) above, 
is stricter than the one suggested by Gilbert \cite{gilbert}
for ordered groupoids.
\end{rem}

\begin{prop}\label{reflexive}
The relation $\sim$ is reflexive.
\end{prop}

\begin{proof}
Suppose that $(g,x) \in \overline{X}$.
Since 
$d(g) \cdot x$ is defined and, by (C1), equal to $x$, 
we get that $(g,x) = (gd(g),x) \sim (g , d(g) \cdot x) = (g,x)$.
\end{proof}

\begin{defi}
Define a partial set action by ${\rm mor}(G)$ 
on $\overline{X}$, in the sense of Definition \ref{setaction}, 
in the following way.
Given $g,h \in {\rm mor}(G)$ and $x \in X$
such that $(h,x) \in \overline{X}$,
then $g \cdot (h,x)$ is defined if and only if $(g,h) \in G^2$.
In that case we put $g \cdot (h,x) = (gh,x)$.
\end{defi}

\begin{prop}\label{twoterms}
If $(g , x),(g' , x') \in \overline{X}$ satisfy $(g,x) \sim (g',x')$,
then, for every $p \in {\rm mor}(G)$ with $(p,g),(p,g') \in G^2$,
we have that $(pg,x) \sim (pg',x')$. 
\end{prop}

\begin{proof}
Case (i): there is $h \in {\rm mor}(G)$
such that $(g',h) \in G^2$, $h \cdot x$ is defined and 
the equalities $g = g' h$ and $x' = h \cdot x$ hold.
Since $(p,g) \in G^2$ we know that also $(p,g') \in G^2$ and
we can write $pg = p g' h$. 
Hence $(pg,x) \sim (pg',x')$. 

Case (ii): $x=x'$, $g,g' \in {\rm ob}(G)$
and both $g \cdot x$ and $g' \cdot x'$ are defined.
Since $(p,g),(p,g') \in G^2$, we get that $g=g'$.
Thus $(pg,x) = (pg',x')$ and hence, 
from Proposition \ref{reflexive}, we get that $(pg,x) \sim (pg',x')$.
\end{proof}

\begin{rem}\label{generated}
Recall that if $R$ is a reflexive relation on a set $S$, 
then the equivalence relation
generated by $R$ equals the set of all $(a,b) \in S \times S$
with the property that there is an integer $n \geq 2$
and $s_1,\ldots,s_n \in S$ such that $s_1 = a$, $s_n = b$,
and for every $i \in \{ 1,\ldots,n-1 \}$,
either $(s_i,s_{i+1}) \in R$ or $(s_{i+1},s_i) \in R$ holds
(see \cite[Proposition 2.15]{joshi}).
\end{rem}

\begin{defi}
Let $\simeq$ denote the equivalence relation on $\overline{X}$
generated by $\sim$ and put $Y = \overline{X} / \simeq$.
The equivalence class of $(g,x)$ in $Y$ is denoted $[g,x]$.
\end{defi}

\begin{prop}\label{isdefined}
Suppose that $(g,x),(g',x') \in \overline{X}$.
If $[g,x]=[g',x']$, then $g \cdot x$ is defined if and only if 
$g' \cdot x'$ is defined. In that case, $g \cdot x = g' \cdot x'$.
\end{prop}

\begin{proof}
From Remark \ref{generated}, it follows that we have a sequence of transitions 
$$(g,x) = (g_1,x_1) \rightarrow (g_2,x_2) \rightarrow
\cdots \rightarrow (g_n,x_n) = (g',x'),$$
where, for each $i \in \{ 1,\ldots,n \}$, 
$(g_i,x_i) \in \overline{X}$, and for every 
$i \in \{ 1,\ldots,n-1 \}$ either
$(g_i,x_i) \sim (g_{i+1},x_{i+1})$ or
$(g_{i+1},x_{i+1}) \sim (g_i,x_i)$.
Using induction over $n$, we prove that
$g \cdot x$ is defined if and only if $g' \cdot x'$ is defined,
and, in that case $g \cdot x = g' \cdot x'$.

Base case: $n=2$. Then $(g,x) \sim (g',x')$ or
$(g',x') \sim (g,x)$. By symmetry it is enough 
to treat the case $(g,x) \sim (g',x')$.
Subcase (i): there is $h \in {\rm mor}(G)$ such that
$(g',h) \in G^2$, $h \cdot x$ is defined,
$g = g' h$ and $x' = h \cdot x$.
By (C3), we get that
$g \cdot x = (g'h) \cdot x$ is defined if and only if
$g' \cdot (h \cdot x) = g' \cdot x'$ is defined,
and, in that case, $g \cdot x = (g'h) \cdot x =
g' \cdot (h \cdot x) = g' \cdot x'$. 
Subcase (ii): $x=x'$, $g,g' \in {\rm ob}(G)$
and both $g \cdot x$ and $g' \cdot x'$ are defined.
From (C1) and (C2) it follows that 
$g \cdot x = x = x' = g' \cdot x'$.

Induction step: suppose that $n > 2$ and that the claim holds for 
all $m < n$.
By the induction hypothesis and the base case, we get that
$g \cdot x$ is defined if and only if $g_{n-1} \cdot x_{n-1}$ 
is defined if and only if $g' \cdot x'$ is defined,
and, in that case, $g \cdot x = g_{n-1} \cdot x_{n-1} = g' \cdot x'$.
\end{proof}

\begin{defi}\label{definitionglobalaction}
Define a set action by ${\rm mor}(G)$ on $Y$ in the following way.
Take $g \in {\rm mor}(G)$ and $(h,x) \in \overline{X}$.
Then $g \cdot [h,x]$ is defined if and only if 
there is $(h',x') \in \overline{X}$ such that
$(h,x) \simeq (h',x')$ and $(g,h') \in G^2$.
In that case, we put $g \cdot [h,x] = [gh',x']$. 
\end{defi}

\begin{prop}
This is a well defined global category action by $G$ on $Y$.
\end{prop}

\begin{proof}
It is clear that (C1)-(C4) of Definition \ref{definition}
hold once we have shown that the action is well defined.
To this end, suppose that $g,h,h' \in {\rm mor}(G)$
and $x,x' \in X$ are chosen so that 
$(h,x),(h',x') \in \overline{X}$,
$(g,h),(g,h') \in G^2$ and
$[h,x] = [h',x']$.
We wish to show that $[gh,x]=[gh',x']$.
From Remark \ref{generated}, it follows that 
we have a sequence of transitions 
$$(h,x) = (h_1,x_1) \rightarrow (h_2,x_2) \rightarrow
\cdots \rightarrow (h_n,x_n) = (h',x'),$$
where, for each $i \in \{ 1,\ldots,n \}$, 
$(h_i,x_i) \in \overline{X}$, and for every 
$i \in \{ 1,\ldots,n-1 \}$ either
$(h_i,x_i) \sim (h_{i+1},x_{i+1})$ or
$(h_{i+1},x_{i+1}) \sim (h_i,x_i)$.

We prove that $[gh,x]=[gh',x']$ by induction over $n$.

Base case: $n=2$.
Then $(h,x) \sim (h',x')$ or 
$(h',x') \sim (h,x)$.
From Proposition \ref{twoterms} it follows that
$(gh,x) \sim (gh',x')$ or $(gh',x') \sim (gh,x)$.
In either case we get that 
$[gh,x]=[gh',x']$.

Induction step: 
suppose that $n > 2$ and that the claim holds for 
all $m < n$.

Case 1: $h' \in {\rm ob}(G)$.
Then $h' = d(g) = c(h)$ and $h' \cdot x'$ is defined.
From Proposition \ref{isdefined} it follows that
$h \cdot x$ is defined and equal to $x'$.
Thus $(h,x) \sim (c(h) , h \cdot x) = (h' , x')$.
From Proposition \ref{twoterms} we get that
$(gh,x) \sim (gh',x')$. Thus $[gh,x]=[gh',x']$.

Case 2: $h' \notin {\rm ob}(G)$.
From Definition \ref{definitionsim}(i), we get that 
$c(h_{n-1}) = c(h_n) = c(h') = d(g)$.
Thus, by the induction hypothesis, we get that
$[gh,x] = [g h_{n-1} , x_{n-1}]$.
From the base case, we get that 
$[g h_{n-1} , x_{n-1}] = [gh',x']$.
Thus $[gh,x] = [gh',x']$.
\end{proof}

\subsection*{Proof of Theorem \ref{maintheorem}}
We wish to show that $Y$ is a universal globalization of $X$.
To this end, define a function $i : X \rightarrow Y$
in the following way.
Take $x \in X$. From (C1) it follows that we can choose
$e \in {\rm ob}(G)$ such that $e \cdot x$ is defined.
Put $i(x) = [e , x]$.

First we show that $i$ is well defined.
Suppose that $e,f \in {\rm ob}(G)$ have the property that
both $e \cdot x$ and $f \cdot x$ are defined.
From the definition of $\sim$ it follows that
$(e,x) \sim (f,x)$. 
Thus $[e,x] = [f,x]$.

Now we show that $i$ is a $G$-function.
Take $x \in X$ and $g \in {\rm mor}(G)$ such that $g \cdot x$ is defined.
First we show that $g \cdot i(x)$ is defined.
Since $g \cdot x$ is defined we get that $d(g) \cdot x$ is defined.
Since $(g , d(g)) \in G^2$, we get that 
$g \cdot [d(g) , x]$ is defined i.e. that
$g \cdot i(x)$ is defined.
Next we show that $i(g \cdot x) = g \cdot i(x)$.
Since $g \cdot x$ is defined and $(c(g),g) \in G^2$,
it follows from (C3) that $c(g) \cdot (g \cdot x)$ is 
defined and equal to $g \cdot x$.
Hence $i(g \cdot x) = [c(g) , g \cdot x]$. 
From the definition of $\sim$ we get that
$(g,x) \sim (c(g) , g \cdot x)$.
Thus $[c(g) , g \cdot x] = [g,x] = g \cdot [c(g) , x] = g \cdot i(x)$.

Now we show that $Y$ is universal.
Suppose that $Z$ is another globalization 
of the partial action by $G$ on $X$.
Suppose that $j : X \rightarrow Z$ is a $G$-function.
Define a map $k : Y \rightarrow Z$ in the following way.
Suppose that $(g,x) \in \overline{X}$. 
Put $k( [g,x] ) = g \cdot j(x)$.
First we show that $g \cdot j(x)$ is defined.
From the definition of $\overline{X}$, we get that
$d(g) \cdot x$ is defined.
Since $j$ is a $G$-function, this implies that 
$d(g) \cdot j(x)$ is defined.
Since $Z$ is global, we get, by (C4),
that $g \cdot j(x)$ is defined.
Now we show that $k$ is well defined.
Suppose that $[g,x] = [g',x']$ for some
$(g,x),(g',x') \in \overline{X}$.
From Remark \ref{generated}, it follows that there is
a sequence of transitions 
$$(g,x) = (g_1,x_1) \rightarrow (g_2,x_2) \rightarrow
\cdots \rightarrow (g_n,x_n) = (g',x'),$$
where, for each $i \in \{ 1,\ldots,n \}$, 
$(g_i,x_i) \in \overline{X}$, and for every 
$i \in \{ 1,\ldots,n-1 \}$ either
$(g_i,x_i) \sim (g_{i+1},x_{i+1})$ or
$(g_{i+1},x_{i+1}) \sim (g_i,x_i)$.
Using induction over $n$, we prove that
$k( [g,x] ) = k( [g',x'] )$.

Base case: $n=2$.
Then either $(g,x) \sim (g',x')$ or $(g',x') \sim (g,x)$.
By symmetry, it is enough to treat the
case $(g,x) \sim (g',x')$.
Subcase (i): there is $h \in {\rm mor}(G)$ such that
$(g',h) \in G^2$, $h \cdot x$ is defined,
$g = g' h$ and $x' = h \cdot x$.
But since $j$ is a $G$-function, we get that
$h \cdot j(x)$ is defined and equal to $j(h \cdot x)$.
Hence, by (C3), we get that 
$k([g,x]) = g \cdot j(x) = (g'h) \cdot j(x) =
g' \cdot (h \cdot j(x)) = g' \cdot j(h \cdot x) = 
g' \cdot j(x') = k([g',x'])$.
Subcase (ii): $x=x'$, $g,g' \in {\rm ob}(G)$
and both $g \cdot x$ and $g' \cdot x'$ are defined.
From (C1) and (C2) it follows that 
$g \cdot x = x = x' = g' \cdot x'$.
Thus $k( [g,x] ) = g \cdot j(x) = j( g \cdot x) = j(x) = j(x') = 
j(g' \cdot x') = g' \cdot j(x') = k( [g',x'] )$.

Induction step: 
suppose that $n > 2$ and that the claim holds for 
all $m < n$.
From the induction hypothesis and the base case,
it follows that $k( [g,x] ) = k( [g_{n-1},x_{n-1}] ) = k([g',x'])$.

Now we show that $k$ is a $G$-function.
Take $(h,x) \in \overline{X}$ and $g \in {\rm mor}(G)$
such that $(g,h) \in G^2$.
Then, since $h \cdot j(x)$ is defined, we get, from (C3), that 
$k( g \cdot [h,x] ) = k( [gh,x] ) =
(gh) \cdot j(x) = g \cdot (h \cdot j(x)) = 
g \cdot ( j( [h,x] ) ) = g \cdot k( [h,x] )$.

Next, we show that $j = k \circ i$. Take $x \in X$
and $e \in {\rm ob}(G)$ such that $e \cdot x$ is defined.
Then, since $j$ is a $G$-function and $e \cdot x$ is defined,
and equal to $x$, we get that 
$(k \circ i)(x) = k( [e,x] ) = e \cdot j(x) =
j( e \cdot x ) = j(x)$.

Finally, we show that $k$ is unique.
Suppose that $k' : Y \rightarrow Z$ is another
$G$-function such that $j = k' \circ i$.
Take $(g,x) \in \overline{X}$. 
Thus, since $k'$ is a $G$-function, we get that
$k'( [g,x] ) = k'( g \cdot [d(g),x] ) =
g \cdot k'( [d(g),x] ) = g \cdot (k' \circ i)(x) =
g \cdot j(x) = k( [g,x] )$.
\hfill $\square$

\begin{rem}\label{injectivecategory}
The map $i : X \rightarrow Y$ from the proof
of Theorem \ref{maintheorem} above is, in fact, {\it injective}
and the global category action by $G$ on $i(X)$ 
induces the original partial category action by $G$ on $X$. 
Indeed, take $x,x' \in X$ and $e,e' \in {\rm ob}(G)$
such that $e \cdot x$ and $e' \cdot x'$ are defined.
Suppose that $i(x)=i(x')$ i.e.
that $[e,x]=[e',x']$.
Since both $e \cdot x$ and $e' \cdot x'$ are defined,
then, from Proposition \ref{isdefined}, we can conclude
that $x = e \cdot x = e' \cdot x' = x'$.
Thus $i$ is injective.
Now take $y,z \in X$ and $g \in {\rm mor}(G)$
such that $g \cdot i(y)$ is defined and equal to $i(z)$.
We wish to show that $g \cdot y$ is defined and equal to $x$.
Take $p,q \in {\rm ob}(G)$ such that $p \cdot y$ 
and $q \cdot z$ are defined. Since $g \cdot i(y)$
is defined, there is, by Definition \ref{definitionglobalaction},
$y' \in X$ and $h \in {\rm mor}(G)$ such that $(g,h) \in G^2$,
$d(h) \cdot y'$ is defined and $i(y) = [p,y] = [h,y']$.
Then we get that $[gh,y'] = g \cdot [h,y'] = g \cdot [p,y] = 
g \cdot i(y) = i(z) = [q , z]$.
But since $q \cdot z$ is defined, we get, from
Proposition \ref{isdefined}, that $(gh) \cdot y'$ also
is defined and equal to $q \cdot z = z$.
However, since $p \cdot y$ is defined, we get, from
Proposition \ref{isdefined}, that $h \cdot y'$
is defined and equal to $p \cdot y = y$.
From (C3), we get that $g \cdot (h \cdot y')$
is also defined and equal to $(gh) \cdot y' = z$.
Thus $g \cdot y = g \cdot (h \cdot y')$ is also defined
and equal to $z$. 
\end{rem}

Now we calculate the universal globalization of $X$, in two cases,
when $G$ is the smallest non-discrete category with two objects.

\begin{exa}\label{example1}
Let $G$ be the category having ${\rm ob}(G) = \{ e , f \}$
and ${\rm mor}(G) = \{ e,f,g \}$ where $g : e \rightarrow f$.
Let the set $X = \{ 1,2,3 \}$ be equipped with a partial 
category action by $G$ defined by the relations
$e \cdot 1 = 1$, $e \cdot 2 = 2$, $f \cdot 2 = 2$, 
$f \cdot 3 = 3$ and $g \cdot 2 = 2.$
Then it is clear that 
$X_e = \{ 1,2 \}$,  $X_f = \{ 2,3 \}$ and $X_g = {}_g X = \{ 2 \}.$
Also 
$$\overline{X} = \{ (e,1) , (e,2) , (f,2) , (f,3) , (g,1) , (g,2) \}.$$
However since
$(e,2) \sim (f,2)$ and $(g,2) = (fg,2) \sim (f, g\cdot 2) = (f,2)$,
we get that 
$$Y = \{ [e,1],[e,2],[f,3],[g,1] \}.$$
The global category action by $G$ on $Y$ is given by
$e \cdot [e,1] = [e,1]$, 
$e \cdot [e,2] = [e,2]$,
$g \cdot [e,1] = [g,1]$,
$g \cdot [e,2] = [e,2]$,
$f \cdot [f,3] = [f,3]$ and 
$f \cdot [g,1] = [g,1]$.
The injective $G$-function $i : X \rightarrow Y$ is defined by
$i(1) = [e,1]$, 
$i(2) = [e,2]$ and
$i(3) = [f,3].$
\end{exa}

Note that, in the previous example, $G$ can be embedded
in a groupoid (see Example \ref{example2}).
On account of Proposition \ref{bijectivesubset}
(see Section \ref{sectionpartialgroupoidactions})
this is not possible for the category in the next example.

\begin{exa}\label{example3}
Let $G$ be the category from Example \ref{example1}.
Let the set $X = \{ 1,2,3,4 \}$ be equipped with a partial 
category action by $G$ defined by the relations
$e \cdot 1 = 1$,
$e \cdot 2 = 2$, 
$e \cdot 3 = 3$,  
$f \cdot 2 = 2$, 
$f \cdot 3 = 3$,
$f \cdot 4 = 4$ and
$g \cdot 2 = g \cdot 3 = 2.$
Then it is clear that 
$X_e = \{ 1,2,3 \}$,
$X_f = \{ 2,3,4 \}$,
$X_g = \{ 2,3 \}$ and 
${}_g X = \{ 2 \}.$
Also 
$$\overline{X} = \{ (e,1) , (e,2) , (e,3) , (f,2) , (f,3) ,
(f,4) , (g,1) , (g,2) , (g,3) \}.$$
However since
$(g,3) = (fg,3) \sim (f,g \cdot 3) = (f,2) \sim (e,2)$,
$(e,3) \sim (f,3)$ and 
$(g,2) = (fg , 2) \sim (f, g \cdot 2) = (f , 2),$
we get that 
$$Y = \{ [e,1] , [e,2] , [e,3] , [f,4] , [g,1] \}.$$
The global category action by $G$ on $Y$ is given by
$e \cdot [e,1] = [e,1]$,
$e \cdot [e,2] = [e,2]$,
$e \cdot [e,3] = [e,3]$,
$f \cdot [e,2] = [e,2]$,
$f \cdot [e,3] = [e,3]$,  
$f \cdot [f,4] = [f,4]$,
$f \cdot [g,1] = [g,1]$,
$g \cdot [e,1] = [g,1]$,
$g \cdot [e,2] = [e,2]$ and
$g \cdot [e,3]=[e,2]$.
The injective $G$-function $i : X \rightarrow Y$ is defined by
$i(1) = [e,1]$,
$i(2) = [e,2]$,
$i(3) = [e,3]$ and 
$i(4) = [f,4].$
\end{exa}

\section{Partial Groupoid Actions}\label{sectionpartialgroupoidactions}

In this section,
we state our conventions on groupoids
and we define partial groupoid actions on sets
(see Definition \ref{definitiongroupoid}).
Then we show that the the partial category actions of a groupoid are precisely
its partial groupoid actions
(see Proposition \ref{categorygroupoid}).
Thereby we generalize \cite[Proposition 2.4]{megrelishvili}.
After that, we show that the definition of partial
groupoid actions can be reformulated in terms
of functions on subsets analogously with 
the axioms (G1$'$)-(G3$'$) from the group case
(see Remark \ref{bijectivesubset}).
Then we generalize \cite[Theorem 3.4]{kellendonk} and show
that mediating $G$-functions between universal globalizations
always are injective (see Proposition \ref{sharpening}).
At the end of the section, we compare our definition with the one
given by Gilbert \cite{gilbert} for partial actions of 
ordered groupoids on sets (see Remark \ref{comparegilbert}).

\subsection*{Conventions on groupoids}
Suppose that $G$ is a category.
Recall that $g \in {\rm mor}(G)$ is called
an {\it isomorphism} if there is $h \in {\rm mor}(G)$,
such that $(g,h) \in G^2$, $gh = d(h)$ and $hg = d(g)$.
In that case we put $g^{-1} = h$.
Recall that $G$ is called a {\it groupoid} if all its 
morphisms are isomorphisms.

\begin{defi}\label{definitiongroupoid}
Suppose that $X$ is a set and $G$ is a groupoid.
By a {\it partial groupoid action by $G$ on $X$}, 
we mean a partial set action by ${\rm mor}(G)$ on $X$,
in the sense of Definition \ref{setaction},
satisfying the following three axioms.
\begin{itemize}

\item[(GR1)] For every $x \in X$, there is $e \in {\rm ob}(G)$
such that $e \cdot x$ is defined. 
If $f \in {\rm ob}(G)$ and $x \in X$ are chosen so that
$f \cdot x$ is defined, then $f \cdot x = x$.

\item[(GR2)] If $x \in X$ and $g \in {\rm mor}(G)$
are chosen so that $g \cdot x$ is defined, then
$g^{-1} \cdot (g \cdot x)$ is defined and equal to $x$.

\item[(GR3)] Suppose that $(g,h) \in G^2$ and $x \in X$ are chosen so that
$g \cdot (h \cdot x)$ is defined, then $(gh) \cdot x$
is defined and equal to $g \cdot (h \cdot x)$.

\end{itemize}
We say that such an action by $G$ on $X$ is {\it global}
if the following axiom holds.
\begin{itemize}

\item[(GR4)] If $g \in {\rm mor}(G)$ and $x \in X$
are chosen so that $d(g) \cdot x$ is defined,
then $g \cdot x$ is defined.

\end{itemize}
\end{defi}

\begin{prop}\label{categorygroupoid}
The partial category actions of a groupoid are precisely
its partial groupoid actions.
\end{prop}

\begin{proof}
Suppose that $G$ is a groupoid and $X$ is a set
equipped with a partial set action by ${\rm mor}(G)$.
We have to show that (C1)-(C3) hold if and only if (GR1)-(GR3) hold. 

Suppose that (C1)-(C3) hold.
The condition (GR1) is identical to (C1).
Now we show (GR2). Suppose that $x \in X$ and 
$g \in {\rm mor}(G)$ are chosen so that $g \cdot x$ is defined.
From (C2), we get that $d(g) \cdot x$ is defined
and equal to $x$.
But then $(g^{-1} g) \cdot x$ is also defined.
Since $(g^{-1},g) \in G^2$, we thus get, from (C3), that
$g^{-1} \cdot (g \cdot x)$ is defined and equal to
$(g^{-1} g) \cdot x = d(g) \cdot x = x$.
Now we show (GR3).
Suppose that $(g,h) \in G^2$ and $x \in X$ are chosen so that
$g \cdot (h \cdot x)$ is defined.
Then $h \cdot x$ is defined.
Thus, from (C3), we get that $(gh) \cdot x$ also is 
defined and equal to $g \cdot (h \cdot x)$.

Now suppose that (GR1)-(GR3) hold.
The condition (C1) is identical to (GR1).
Now we show (C2).
Take $x \in X$ and $g \in {\rm mor}(G)$
so that $g \cdot x$ is defined. 
From (GR2), we get that $g^{-1} \cdot (g \cdot x)$
is defined and equal to $x$.
From (GR3), we get that $(g^{-1} g) \cdot x$ also is defined
and equal to $g^{-1} \cdot (g \cdot x)$.
Putting this together shows that 
$d(g) \cdot x$ is defined and equal to $x$.
Now we show (C3).
Suppose that $(g,h) \in G^2$ and $x \in X$ are chosen so that
$h \cdot x$ is defined. 
Suppose first that $g \cdot (h \cdot x)$ is defined.
Then, from (GR3), it follows that $(gh) \cdot x$
is defined and equal to $g \cdot (h \cdot x)$.
Now suppose that $(gh) \cdot x$ is defined.
We need to show that $g \cdot (h \cdot x)$ is defined.
From (GR2), it follows that $h^{-1} \cdot (h \cdot x)$
is defined and equal to $x$.
Thus $(gh) \cdot (h^{-1} \cdot (h \cdot x))$ is defined.
From (GR3), it follows that 
$(g h h^{-1}) \cdot (h \cdot x)$
is defined. But since $ghh^{-1} = g$, 
we get that $g \cdot (h \cdot x)$ is defined.
\end{proof}

\begin{prop}\label{inverse}
If $G$ is a groupoid and $X$ is a set equipped
with a partial groupoid action by $G$, then,
using the notation from Section \ref{sectionpartialcategoryactions},
we get that for each $g \in {\rm mor}(G)$,
the function $\alpha_g$ is bijective with
$(\alpha_g)^{-1} = \alpha_{g^{-1}}$.
\end{prop}

\begin{proof}
Take $g \in {\rm mor}(G)$.
By the definition of the function 
$\alpha_g$ it follows that it is surjective.
Now we show that $\alpha_g$ is injective.
From (GR2) it follows that 
${}_g X \subseteq X_{g^{-1}}$ and that 
$\alpha_{g^{-1}} ( \alpha_g (x) ) = x$, for $x \in X_g$.
Thus $\alpha_g$ is also injective.
What remains to show is that ${}_g X \supseteq X_{g^{-1}}$.
Take $x \in X_{g^{-1}}$.
Then, since ${}_{g^{-1}} X \subseteq X_g$, by (GR2),
we get that $x = \alpha_g ( \alpha_{g^{-1}} (x) ) \in {}_g X$.
\end{proof}

\begin{rem}\label{bijectivesubset}
It is clear from Proposition \ref{alternative},
Proposition \ref{alternative'} and Proposition \ref{inverse}
that Definition \ref{definitiongroupoid} can be 
reformulated in terms of functions defined on subsets 
of $X$ in the following way.
A partial groupoid action by $G$ on $X$ is a pair
$( \{ X_g \}_{g \in {\rm mor}(G)} , \{ \alpha_g \}_{g \in {\rm mor}(G)} )$,
where for each $g \in {\rm mor}(G)$, $X_g$ is a subset of $X$ and
$\alpha_g : X_g \rightarrow X_{g^{-1}}$ is a bijection
satisfying the following three axioms.

\begin{itemize}

\item[(GR1$'$)] $X = \cup_{e \in {\rm ob}(G)} X_e$ and
$\alpha_e = {\rm id}_{X_e}$, for $e \in {\rm ob}(G)$.

\item[(GR2$'$)] If $g \in {\rm mor}(G)$, then $X_g \subseteq X_{d(g)}$.

\item[(GR3$'$)] Suppose that $(g,h) \in G^2$. Then 
$\alpha_h( X_h \cap X_{gh} ) = X_g \cap X_{h^{-1}}$.
If $x \in X_h \cap X_{gh}$, then 
$\alpha_g(\alpha_h (x)) = \alpha_{gh}(x)$.

\end{itemize}
If $G$ is a group and we for every $g \in G$, put 
$D_g = X_{g^{-1}}$, then it is easy to see that
the axioms (GR1$'$)-(GR3$'$) coincide with the 
axioms (G1$'$)-(G3$'$). 
\end{rem}

\begin{rem}\label{sharpening}
Theorem \ref{maintheorem} can be slightly sharpened
in the case of partial groupoid actions.
Namely, suppose that $G$ is a groupoid and $X$ is a set
equipped with a partial groupoid action by $G$.
Let $Y$ be the universal globalization of $X$,
as it was defined in Section \ref{sectionglobalization},
with corresponding injective $G$-function $i : X \rightarrow Y$.
Suppose that $Z$ is another globalization 
of the partial action by $G$ on $X$
with corresponding injective $G$-function $j : X \rightarrow Z$.
Define the unique $G$-function $k : Y \rightarrow Z$,
subject to the relation $k \circ i = j$,
as in Section \ref{sectionglobalization}.
Analogous to the group case \cite[Theorem 3.4]{kellendonk},
we can show that $k$ is {\it injective}.
To this end, take $g,h \in {\rm mor}(G)$ and $x,x' \in X$
such that $(g,x),(h,x') \in \overline{X}$
and $k( [g,x] ) = k( [h,x'] )$. This implies that 
$g \cdot j(x) = h \cdot j(x')$.
From (GR2) and (GR3), we get that
$h^{-1} \cdot (g \cdot j(x))$ is defined and equal to $j(x')$.
But then there must exist $(g_1,x_1) \in \overline{X}$
such that $c(h) = c(g_1)$ and $[g,x] = [g_1,x_1]$.
Thus, from (GR3), we get that
$j(x') = h^{-1} \cdot (g \cdot j(x)) = (h^{-1} g_1) \cdot j(x_1)$.
From Remark \ref{injectivecategory}, we get that
$(h^{-1} g_1) \cdot x_1$ is defined and equal to $x'$.
Thus from (GR3), we get that $g_1 \cdot x_1$ 
is defined and equal to $h \cdot x'$.
But from Proposition \ref{isdefined}, we get that
$g \cdot x$ also is defined and equal $g_1 \cdot x_1$.
Thus $g \cdot x = h \cdot x'$.
From this we get that
$(g,x) \sim (c(g) , g \cdot x) \sim (c(h) , h \cdot x') \sim (h,x')$.
Thus $[g,x] = [h,x']$.
\end{rem}

\begin{rem}\label{comparegilbert}
As mentioned in the introduction, 
Gilbert \cite{gilbert} has shown a globalization
result analogous to our Theorem \ref{maintheorem}
in the case when the category $G$ is an {\it ordered} groupoid.
However, he is only able to show that mediating $G$-functions
are partially injective (see \cite[Proposition 4.10]{gilbert}).
This depends on the fact that our relation $\sim$
is stricter (see Remark \ref{comparesim})
than the one defined in loc. cit. and, hence,
our equivalence relation $\simeq$
is also stricter than Gilbert's
(see \cite[p. 190]{gilbert}).
Therefore we have fewer equivalence classes in $Y$
which makes all our mediating $G$-functions injective.
For a concrete example displaying this difference,
see Example \ref{example4} below.
\end{rem}

Now we calculate the universal globalization of $X$, in two cases,
when $G$ is the smallest non-discrete groupoid with two objects.

\begin{exa}\label{example2}
Let $G$ be the groupoid having ${\rm ob}(G) = \{ e , f \}$
and ${\rm mor}(G) = \{ e,f,g,g^{-1} \}$ where $g : e \rightarrow f$
(so that $g^{-1} : f \rightarrow e$).
Let the set $X = \{ 1,2,3 \}$ be equipped with a partial 
groupoid action by $G$ defined by the relations
$e \cdot 1 = 1$,
$e \cdot 2 = 2$,
$f \cdot 2 = 2$, 
$f \cdot 3 = 3$,
$g \cdot 2 = 2$, and
$g^{-1} \cdot 2 = 2.$
Then it is clear that 
$X_e = \{ 1,2 \}$,
$X_f = \{ 2,3 \}$ and 
$X_g = {}_g X = X_{g^{-1}} = {}_{g^{-1}} X = \{ 2 \}.$
Also 
$\overline{X} = \{ (e,1) , (e,2) , (f,2) , (f,3) , (g,1) , (g,2),
(g^{-1},2) , (g^{-1},3) \}.$
However since
$(e,2) \sim (f,2)$,
$(g,2) = (fg,2) \sim (f, g\cdot 2) = (f,2)$ and
$(e,2) = (g g^{-1} , 2) \sim (g^{-1} , g \cdot 2) = (g^{-1} , 2),$
we get that 
$$Y = \{ [e,1],[e,2],[f,3],[g,1],[g^{-1},3] \}.$$
The global category action by $G$ on $Y$ is given by
$e \cdot [e,1] = [e,1]$, 
$e \cdot [e,2] = [e,2]$,
$g \cdot [e,1] = [g,1]$,
$g \cdot [e,2] = [e,2]$, 
$f \cdot [f,3] = [f,3]$, 
$f \cdot [g,1] = [g,1]$,
$g^{-1} \cdot [e,2] = [e,2]$, 
$g^{-1} \cdot [g,1] = [e,1]$ and
$g^{-1} \cdot [f,3] = [g^{-1},3]$.
The injective $G$-function $i : X \rightarrow Y$ is defined by
$i(1) = [e,1]$, 
$i(2) = [e,2]$ and 
$i(3) = [f,3].$
\end{exa}

\begin{exa}\label{example4}
Let $G$ be the groupoid from Example \ref{example2}.
Let the set $X = \{ 1,2,3 \}$ be equipped with a partial 
groupoid action defined by the relations
$e \cdot 1 = 1$,
$e \cdot 2 = 2$,
$e \cdot 3 = 3$, 
$f \cdot 2 = 2$, 
$f \cdot 3 = 3$, 
$g \cdot 1 = 2$, 
$g \cdot 2 = 3$, 
$g^{-1} \cdot 2 = 1$ and 
$g^{-1} \cdot 3 = 2.$
Note that this is the same example of partial groupoid actions 
as in Gilbert \cite[Example 4.2]{gilbert}.
It is clear that 
$X_e = \{ 1,2,3 \}$,
$X_f = \{ 2,3 \}$, 
$X_g = \{ 1,2 \}$ and 
$X_{g^{-1}} = \{ 2,3 \}.$
Also 
$$\overline{X} = \{ (e,1) , (e,2) , (e,3) , (f,2) , (f,3) ,
(g,1) , (g,2) , (g,3) , (g^{-1},2) , (g^{-1},3) \}.$$
However since
$(g,1) = (fg,1) \sim (f,g \cdot 1) = (f,2) \sim (e,2)$,
$(g^{-1},3) = (e g^{-1},3) \sim (e , g^{-1} \cdot 3) = (e,2)$ and
$(g,2) = (fg,2) \sim (f ,g \cdot 2) = (f , 3) \sim (e,3)$,
we get that 
$$Y = \{ [e,1] , [e,2] , [e,3] , [g,3] \}.$$
Note that, with the definition of universal globalization 
that Gilbert \cite[Example 4.2]{gilbert} suggests, then
we get as many as {\it six} equivalence classes in the quotient space,
whereas we only get {\it four}.
The global category action by $G$ on $Y$ is given by
$e \cdot [e,1] = [e,1]$
$e \cdot [e,2] = [e,2]$,
$e \cdot [e,3] = [e,3]$,
$f \cdot [e,2] = [e,2]$,
$f \cdot [e,3] = [e,3]$,
$f \cdot [g,1] = [g,1]$,
$g \cdot [e,1] = [e,2]$,
$g \cdot [e,2] = [e,3]$ and
$g \cdot [e,3]=[g,3]$.
The injective $G$-function $i : X \rightarrow Y$ is defined by
$i(1) = [e,1]$,
$i(2) = [e,2]$ and 
$i(3) = [e,3].$
\end{exa}

\section{Continuous Partial Actions}\label{sectioncontinuous}

Throughout this section, $X$ denotes a topological space
and $G$ denotes a category which acts partially on $X$.
In this section, we state our conventions on 
partial continuous functions and topological categories,
and we define what it should mean for a 
topological category $G$ to act continuously partially 
on a topological space $X$
(see Definition \ref{definitioncontinuousaction}).
Then we show Theorem \ref{continuousmaintheorem}.
At the end of this section, we also show that if the category is star open
and the action is graph open (see Definition \ref{definitionstaropen}),
then the embedding of $X$, into a universal globalization,
is an open map (see Proposition \ref{open}).

\subsection*{Conventions on continuous partial functions}
Suppose that $A$ and $B$ are topological spaces and that 
$f : A \rightarrow B$ is a partial function, defined on $C \subseteq A$.
We say that $f$ is continuous as a partial map if the
function $f : C \rightarrow B$ is continuous, where
we let $C$ have the relative topology induced from $A$.

\subsection*{Conventions on topological categories}
We say that $G$ is a {\it topological category} if 
${\rm mor}(G)$ is a topological space and the
partial composition
${\rm mor}(G) \times {\rm mor}(G) \rightarrow {\rm mor}(G)$
is continuous. Here we let 
${\rm mor}(G) \times {\rm mor}(G)$ be equipped
with the product topology.

\vspace{2mm}

We assume, from now on, that $G$ is a topological category.

\begin{rem}
The term ``topological category'' has been used in 
completely different senses e.g in 
\cite{brummer}, 
\cite{ehresmann} and 
\cite{holmann}
which we do not consider here.
The term ``topological groupoid'' has been used
by authors in more or less the same sense that 
we do with the only exception that often
it is assumed that ${\rm ob}(G)$ is a topological space
and the maps $d : {\rm mor}(G) \rightarrow {\rm ob}(G)$
and $d : {\rm mor}(G) \rightarrow {\rm ob}(G)$ are assumed
to be continuous, see e.g. \cite{kirill}.
\end{rem}

\begin{defi}\label{definitioncontinuousaction}
We say that $G$ is a {\it continuous partial category action}
on $X$ if the following two axioms hold.
\begin{itemize}

\item[(CA1)] If $e \in {\rm ob}(G)$, then 
$X_e = \{ x \in X \mid e \cdot x \ {\rm is} \ {\rm defined} \}$ 
is open in $X$.

\item[(CA2)] The partial action
${\rm mor}(G) \times X \rightarrow X$
is continuous. Here we let 
${\rm mor}(G) \times X$ be equipped with 
the product topology.

\end{itemize}
\end{defi}

\begin{defi}\label{continuousequivariant}
Suppose that $Y$ is another topological space, 
equipped with a continuous partial action by $G$,
and there is a continuous $G$-function $i : X \rightarrow Y$.
If the action by $G$ on $Y$ is global, then 
$Y$ is called a {\it topological globalization} of $X$.
We say that such a globalization is {\it universal}
if for every continuous $G$-function $j : X \rightarrow Z$,
where $Z$ is a topological space equipped with a global continuous category 
action by $G$, there is a unique continuous $G$-function
$k : Y \rightarrow Z$ such that $j = k \circ i$.
\end{defi}

\subsection*{Proof of Theorem \ref{continuousmaintheorem}}

We wish to show that $Y$ is a universal globalization of $X$.
All the work has already been done in the previous sections.
What remains are the ``topological'' parts.

First of all, let $\overline{X}$ have the relative topology
induced from the product topology on ${\rm mor}(G) \times X$.
Then the partial set action 
${\rm mor}(G) \times \overline{X}
\rightarrow \overline{X}$, denoted by $\beta$, is continuous.
Indeed, put 
$\Gamma' = \{ (g , (h,x)) \in {\rm mor}(G) \times \overline{X} \mid d(g) = c(h) \}$
and let $\Gamma'$ have the relative topology induced
from the product topology on ${\rm mor}(G) \times {\rm mor}(G) \times X$.
We need to show that the function $\beta : \Gamma' \rightarrow \overline{X}$
is continuous.
Take an open $U \subseteq {\rm mor}(G)$ and an open $V \subseteq X$.
Let $m$ denote the multiplication $G^2 \rightarrow {\rm mor}(G)$.
Then $\beta^{-1}( (U \times V) \cap \overline{X} ) = 
( m^{-1}(U) \times V ) \cap \Gamma'$ 
which is open in $\Gamma'$, since $m$ is continuous.

Let $Y$ have the quotient topology induced from 
the topology from $\overline{X}$.
Recall that this means that a subset $U$ of $Y$
is open if and only if $q^{-1}(U)$ is an open subset of $\overline{X}$, 
where $q$ denotes the quotient map $\overline{X} \rightarrow Y$.
Now we show that the partial set action
of ${\rm mor}(G)$ on $Y$ is continuous.
Define an equivalence relation $\equiv$ on $\Gamma'$
by saying that $(g , (h,x)) \equiv (g' , (h',x'))$ when
$g = g'$ and $(h,x) \simeq (h',x')$.
Put $\Gamma'' = \Gamma' / \equiv$
and equipp $\Gamma''$ with the quotient 
topology. 
Since $q$ and $\beta$ are continuous, we get that
$q \circ \beta : \Gamma' \rightarrow Y$ is continuous.
Also, since the action of ${\rm mor}(G)$ on $Y$
is well defined, we get that $q \circ \beta$ respects $\equiv$. 
Thus, by the universal property of $\equiv$,
there is a unique continuous map
$\beta' : \Gamma'' \rightarrow Y$
such that $\beta'[ (g , (h,x) ] = [ gh,x ]$,
for $(g,h) \in G^2$ and $x \in X_{d(h)}$.
Thus, the partial set action
of ${\rm mor}(G)$ on $Y$ is continuous. 

Now we show that the function
$i : X \rightarrow Y$, as it was defined in 
Section \ref{sectionglobalization}, is continuous.
First define a relation $E : X \rightarrow {\rm mor}(G)$
by $E(x) = \{ e \in {\rm ob}(G) \mid e \cdot x \ {\rm is} \ {\rm defined} \}$,
for $x \in X$.
Next, define a relation $R : X \rightarrow \overline{X}$
by $R(x) = E(x) \times \{ x \}$, for $x \in X$.
Then $i = q \circ R$ (composition of relations).
Since $q$ is continuous, we only need to show that $R$
is a continuous relation (in the sense that
inverse images of open sets are open).
To this end, take an open $U$ in ${\rm mor}(G)$
and an open $V$ in $X$. Then 
$R^{-1}(U \times V) = 
\{ x \in X \mid E(x) \subseteq U \ {\rm and} \ x \in V \} =
\{ x \in V \mid E(x) \subseteq U \} = 
V \cap E^{-1}(U \cap {\rm ob}(G)) =
\cup_{e \in U \cap {\rm ob}(G)} V \cap E^{-1}(e) = 
\cup_{e \in U \cap {\rm ob}(G)} V \cap X_e$ which,
by (CA1), is open.

Finally, we show that if $j : X \rightarrow Z$
is another topological globalization of $X$,
then the map $k : Y \rightarrow Z$, defined by
$k([g,x]) = g \cdot j(x)$, for $[g,x] \in Y$, is continuous.
By the universal property of the quotient topology on $Y$,
it follows that it is enough to show that 
the partial function $K : \overline{X} \rightarrow Z$,
defined by $K( (g,x) ) = g \cdot j(x)$, for $(g,x) \in \overline{X}$,
is continuous. To this end,
first define the partial function
$L : \overline{X} \rightarrow {\rm mor}(G) \times Z$
by $L( ( g,x) ) = (g , j(x))$, for $(g,x) \in \overline{X}$.
Then $L$ is continuous. Indeed, take an open $U \subseteq {\rm mor}(G)$
and an open $V \subseteq Y$. Then
$L^{-1}(U \times V) = (U \times j^{-1}V) \cap \overline{X}$
is open in $\overline{X}$.
Next, let $K$ denote the continuous partial category
action ${\rm mor}(G) \times Z \rightarrow Z$.
Since $k = K \circ L$, we get that $k$ is continuous.
\hfill $\square$

\begin{defi}\label{definitionstaropen}
We say that $G$ is {\it star open} if for every $e \in {\rm ob}(G)$,
the set $d^{-1}(e)$ is open in ${\rm mor}(G)$.
We say that the partial action of $G$ on $X$ is
{\it graph open} if the set
$\Gamma = \{ (g,x) \in {\rm mor}(G) \times X \mid
g \cdot x \ {\rm is \ defined} \}$ is open in ${\rm mor}(G) \times X$.
\end{defi}

\begin{rem}
If $G$ is a group (or monoid), then $G$ is always star open.
Indeed, if $e$ denotes the identity element of $G$,
then $d^{-1}(e) = G$.
\end{rem}

\begin{prop}\label{open}
If $G$ is star open and the partial action of $G$ on 
$X$ is graph open, then $i : X \rightarrow Y$ is open.
\end{prop}

\begin{proof}
Take an open subset $U$ of $X$.
By the definition of the quotient topology on $Y$,
we need to show that $q^{-1}(i(U))$ is open in $\overline{X}$.
To this end, note that
$q^{-1}( i(U) ) = \alpha^{-1}(U)$
which is open in $\Gamma$ since, by (CA2), $\alpha$ is continuous.
Since the action is graph open, we get that $\Gamma$ is open in ${\rm mor}(G) \times X$.
Thus $q^{-1}( i(U) )$ is also open in 
${\rm mor}(G) \times X$.
However, since $\overline{X} = 
\cup_{e \in {\rm ob}(G)} d^{-1}(e) \times X_e$,
and $G$ is star open in ${\rm mor}(G)$,
we get, from (CA1), that $\overline{X}$ is open in 
${\rm mor}(G) \times X$.
Thus $q^{-1}( i(U) )$ is open in $\overline{X}$.
\end{proof}

\end{document}